\documentclass[aop,MSNbibl,seceqn,dvips]{arximspdf}

\doi{10.1214/12-AOP771} 
\volume{41}
\issue{6}
\pubyear{2013}
\firstpage{4287}
\lastpage{4316}

\makeatletter

\newtheorem{lemma}{Lemma}[section]
\newtheorem{prop}[lemma]{Proposition}
\newtheorem{theorem}[lemma]{Theorem}
\newtheorem{cor}[lemma]{Corollary}
\newtheorem{con}[lemma]{Conjecture}

\newproclaim{rem}[lemma]{Remark}
\newproclaim{exam}[lemma]{Example}

\newcommand{\rrVert}{\Vert}
\newcommand{\rrvert}{\vert}
\newcommand{\llVert}{\Vert}
\newcommand{\llvert}{\vert}

\newcommand{\nz}{{\mathbb N}}
\newcommand{\rz}{{\mathbb R}}
\newcommand{\cz}{{\mathbb C}}
\newcommand{\ten}{\otimes}

\newcommand{\Om}{\Omega}
\newcommand{\al}{\alpha}
\newcommand{\Si}{\Sigma}
\newcommand{\tet}{\theta}
\newcommand{\la}{\lambda}
\newcommand{\eps}{\varepsilon}
\newcommand{\M}{{\mathcal M}}

\newcommand{\N}{{\mathcal N}}
\newcommand{\ez}{{\mathbb E}}

\makeatother

\begin{document}
\begin{frontmatter}

\title{Noncommutative Bennett and Rosenthal inequalities}
\runtitle{Noncommutative Bennett and Rosenthal inequalities}

\begin{aug}
\author[A]{\fnms{Marius} \snm{Junge}\thanksref{t2}\ead[label=e1]{junge@math.uiuc.edu}}
\and
\author[A]{\fnms{Qiang} \snm{Zeng}\corref{}\ead[label=e2]{zeng8@illinois.edu}}
\runauthor{M. Junge and Q. Zeng}
\affiliation{University of Illinois at Urbana-Champaign}
\address[A]{Department of Mathematics\\
University of Illinois\\
Urbana, Illinois 61801\\
USA\\
\printead{e1}\\
\hphantom{E-mail: }\printead*{e2}} 
\end{aug}

\thankstext{t2}{Supported in part by NSF Grant DMS-09-01457.}

\received{\smonth{11} \syear{2011}}
\revised{\smonth{3} \syear{2012}}

%
\begin{abstract}
In this paper we extend the Bernstein, Prohorov and Bennett
inequalities to the noncommutative setting. In addition we provide an
improved version of the noncommutative Rosenthal inequality,
essentially due to Nagaev, Pinelis and Pinelis, Utev for commutative
random variables. We also present new best constants in Rosenthal's
inequality. Applying these results to random Fourier projections, we
recover and elaborate on fundamental results from compressed sensing,
due to Candes, Romberg and Tao.
\end{abstract}

%

\end{frontmatter}

\setcounter{section}{-1}

\section{Introduction}\label{sec0}

Rosenthal's inequality \cite{Ro} was initially discovered to construct
some new Banach spaces. However, Rosenthal's inequality gives a very
nice bound for the $p$-norm of independent random variables and has
found many generalizations and applications. The martingale version of
Rosenthal's inequality was discovered almost simultaneously by
Burkholder \cite{Bu}. Since then, the order of the constants in these
inequalities has been studied extensively, in particular by Johnson,
Schechtman and Zinn \cite{JSZ}. The correct order in the martingale
version has been established by Hitczenko \cite{Hi}, based on
fundamental work of Kwapie{\'n} and Woyczy{\'n}ski \cite{KW}. Nowadays,
easy proofs of Rosenthal inequalities can be found with the help of
Bernstein, Prohorov and Bennett's inequalities; see \cite{Pr,Be} and
the references therein. Historically, Bernstein's inequality was first
established in the 1920s, according to the references in \cite{Be}.
Later on, Prohorov improved Bernstein's inequality in \cite{Pr}. Then,
Bennett, who seemed to be unaware of Prohorov's work, strengthened
Bernstein's results directly in \cite{Be}, which provided an even more
precise bound than Prohorov's inequality. We will extend Bennett's
inequalities to the noncommutative setting, and then obtain the
noncommutative Bernstein and Prohorov inequalities as consequences.

Let us recall that the classical Rosenthal inequality says that for
independent mean $0$ random variables, we have
%
\begin{equation}
\label{RS1} \Biggl( \ez\Biggl|\sum_{k=1}^n
f_k\Biggr|^p \Biggr)^{1/p} \le c(p) \Biggl( \Biggl(
\sum_{k=1}^n \ez|f_k|^2
\Biggr)^{1/2}+ \Biggl(\sum_{k=1}^n
\ez|f_k|^p \Biggr)^{1/p} \Biggr).
\end{equation}
According to \cite{JSZ}, the order of the best constant here is
$c(p)=p/(1+\log p)$. In this paper we separate the two terms and ask for
%
\begin{equation}
\label{RSAP} \Biggl( \ez\Biggl|\sum_{k=1}^n
f_k\Biggr|^p \Biggr)^{1/p} \le A(p) \Biggl(\sum
_{k=1}^n \ez|f_k|^2
\Biggr)^{1/2} + B(p) \Biggl(\sum_{k=1}^n
\ez|f_k|^p \Biggr)^{1/p}.
\end{equation}
The central limit theorem immediately implies $A(p)\ge c \sqrt{p}$ for
every choice of $B(p)$. Problem (\ref{RSAP}) is by no means new. Nagaev
and Pinelis \cite{NP} obtained a very precise bound on the tail
behavior of $S_n=\sum_{k=1}^n X_k$ which implies that
$(A(p),B(p))=C(\sqrt{p},p)$ is possible. 
Pinelis and Utev showed that in some sense $A(p)=C\sqrt{p}$ and
$B(p)=Cp$ are also best. In Section \ref{sec3}, we will revisit this
problem and show that assuming $A(p)\le C p^m$ for some $m>1/2$, we
must have
\[
B(p) \ge c \frac{p}{1+\log p}.
\]
This is exactly consistent with $(A(p),B(p))=C({p}/(1+\log
p),{p}/(1+\log p))$. Moreover, we show that the worst case is obtained
for independent random selectors $f_k=(\delta_k-\la)$ with expectation
$\la>0$.

We will prove a vast generalization of (\ref{RSAP}) in the
noncommutative setting for conditionally independent random variables
with $A(p)=c\sqrt{p}$ and $B(p)=Cp$. This improves the corresponding
results from \cite{JX3} of the form $A(p)=B(p)=Cp$. Our new results are
motivated by applications in compressed sensing for random selectors
with matrix valued coefficients. More precisely, we have to consider
rank-one operators
\[
a_j =\bigl[\bar{x}_j(l)x_j(r)
\bigr]_{1\le l,r\le n}
\]
such that $|x_k(j)|\le D$. Then the aim is to estimate
%
\begin{equation}
\label{cseq} \Biggl\llVert\frac{1}{k}\sum_{j=1}^n
\delta_j fa_jf-f\Biggr\rrVert_{B(\ell_2^n)} \le\mbox{?}
\end{equation}
for independent selectors $\delta_j\in\{0,1\}$ with $\ez\delta_j=k/n$
and a projection $f$. As in the foundational paper on compressed
sensing by Candes, Romberg and Tao \cite{CRT}, it is tempting to use
moment estimates, or equivalently, estimates of the Schatten $p$-norm
of these matrices. In fact, the improved Rosenthal inequality allows us
to recover the famous estimates in \cite{CRT}.\eject

Let us recall that the noncommutative $L_p$ space associated with the
trace on $B(\ell_2)$ is given by
\[
\|x\|_p =\bigl[\operatorname{tr}\bigl(|x|^p\bigr)\bigr]^{1/p}
= \biggl(\sum_{j} s_j(x)^p
\biggr)^{1/p},
\]
where the singular number $s_j(x)=\la_j(|x|)$, that is, the eigenvalues
of the positive matrix $|x|=\sqrt{x^*x}$. Thus a good estimate of
(\ref{cseq}) can certainly be obtained from an estimate of the form
%
\begin{eqnarray}
\label{ros} \biggl( \ez\biggl\|\sum_j
\delta_j fa_jf-kf \biggr\|_p^p
\biggr)^{1/p} &\le& C \sqrt{p} \biggl\|\sum_j
\ez\bigl(\delta_j^2\bigr) fa_jf^2
\biggr\|_{p/2}^{1/2} \nonumber\\[-9pt]\\[-9pt]
&&{}+ Cp \biggl( \sum_j
\|fa_jf\|_p^p \biggr)^{1/p}.\nonumber
\end{eqnarray}
Let us now describe the more general setup which allows us to prove
results in noncommutative probability which includes all the statements
above. Indeed, we assume that $\M$ is a von Neumann algebra equipped
with a normal faithful tracial state $\tau\dvtx \M\to\cz$, that is,
$\tau
(1)=1$ and $\tau(xy)=\tau(yx)$. Then $L_p(\M,\tau)$ is the completion
of $\M$ with respect to $\|x\|_p=[\tau(|x|^p)]^{1/p}$. It is well known
(see, e.g., \cite{FK,PX}) that $\|\cdot\|_p$ is a norm for $1\le p\le
\infty$. In particular, $\|\cdot\|_\infty=\|\cdot\|$. Here and in the
following, $\|\cdot\|$ will always denote the operator norm. Let
$\mathcal{N}\subset\mathcal{M}$ be a von Neumann subalgebra. Then
there exists a unique conditional expectation $E_\N\dvtx \mathcal{M}\to
\mathcal{N}$ such that $E_\N(1)=1$ and\looseness=-1
\[
E_\N(axb) = a E_\N(x)b,\qquad a,b\in\mathcal{N} \mbox{ and }
x\in\mathcal{M}.
\]\looseness=0
We say that two subalgebras $\mathcal{N}\subset A,B\subset\mathcal{M}$
are \textit{independent} over $\mathcal{N}$ if
\[
E_\N(ab) = E_\N(a)E_\N(b),\qquad a\in A, b\in
B.
\]
In particular, we say that $x, y\in\M$ are independent if the algebras
they generate, respectively, are independent over $\mathbb{C}$. A
sequence of subalgebra $A_1,\ldots,A_n$ are called successively
independent over $\mathcal{N}$ if $A_{k+1}$ is independent of the
algebra $\M(k)$ generated by $A_1,\ldots,A_k$. Our noncommutative Bennett
inequality reads as follows.
%
\begin{theorem}\label{Ben} Let $\mathcal{N}\subset A_j\subset
\mathcal
{M}$ be successively independent over $\N$ and $a_j\in A_j$ be
self-adjoint such that:\vspace*{6pt}

\textup{(i)} $E_\N(a_j)=0$; \textup{(ii)} $E_\N(a_j^2)\le\sigma_j^2$;
\textup{(iii)} $\|a_j\|\le M_j$.\vspace*{6pt}

\noindent Then for $t\ge0$,
\[
\tau\Biggl(1_{[t,\infty)} \Biggl(\sum_{j=1}^n
a_j \Biggr) \Biggr) \le\exp\biggl(-\frac{\sum_{j=1}^n \sigma_j^2}{\sup
_{j=1,\ldots,n}
M_j^2}\phi\biggl(
\frac{t\sup_{j=1,\ldots,n} M_j}{\sum_{j=1}^n
\sigma_j^2} \biggr) \biggr),
\]
where $\phi(x)=(1+x)\log(1+x)-x$.
\end{theorem}\eject

Here we used $1_{I}(a)=\int_I dE_t$ for the spectral projection given
by the spectral decomposition $a=\int t\,dE_t$. We should mention that
the key new ingredient in this theorem is the Golden--Thompson
inequality, which has already played a crucial role in Ahlswede and
Winter's paper \cite{AW}, Gross's paper \cite{Gr} and Oliveira's
paper~\cite{Ol}. The best constants for random matrices probability
inequalities so far are due to Tropp \cite{Tr} by using Lieb's theorem
\cite{Li}. However, it seems Lieb's theorem does not apply to the fully
noncommutative setting. In our approach we allow general randomness via
independence not necessarily given by classical filtrations. Indeed,
all the other works we mentioned only considered the semicommutative
case or the random matrix case where operators with classical
randomness act on a finite-dimensional Hilbert space. We invite the
reader to rewrite the inequality for conditionally independent copies
$x_j$ with $\sigma=\sigma_j$, $M_j=M$. Note that in the commutative
context,
\[
\tau\bigl(1_{[t,\infty)}(a)\bigr)=\operatorname{Prob}(a\ge t).
\]
In the future we will simply take this formula as a definition. Then
our Bernstein and Prohorov inequalities for noncommutative random
variables reads as follows.
%
\begin{cor}\label{Bern} Under the same hypothesis of Theorem \ref{Ben},
we have
%
\begin{equation}
\label{bern} \operatorname{Prob} \Biggl(\sum_{j=1}^n
a_j \ge t \Biggr) \le\exp\biggl(-\frac{t^2}{2\sum_{j=1}^n \sigma
_j^2+ (2t/3)\sup_{j=1,\ldots,n} M_j} \biggr)
\end{equation}
and
%
\begin{eqnarray}
\label{proh}
&&
\operatorname{Prob} \Biggl(\sum_{j=1}^n
a_j \ge t \Biggr) \nonumber\\[-8pt]\\[-8pt]
&&\qquad\le\exp\biggl(-\frac{t}{2\sup_{j=1,\ldots,n}
M_j}
\operatorname{arcsinh} \biggl(\frac{t\sup_{j=1,\ldots,n} M_j}{2\sum_{j=1}^n \sigma
_j^2} \biggr) \biggr).\nonumber
\end{eqnarray}
\end{cor}

It is now rather standard to derive Rosenthal's inequality from
Bernstein's inequality (\ref{bern}).
%
\begin{cor}\label{Ros} Let $2\le p<\infty$ and $a_j$ satisfy the
hypothesis of Theorem~\ref{Ben}. Then
\[
\Biggl\llVert\sum_{j=1}^n
a_j\Biggr\rrVert_p \le C \Biggl( \Biggl(p\sum
_{j=1}^n \sigma_j^2
\Biggr)^{1/2} + p \sup_{j=1,\ldots,n} M_j \Biggr).
\]
\end{cor}

For unbounded operators and fixed $p$, we can prove a similar
inequality. Here we have to make a slightly stronger assumption. Let us\vadjust{\goodbreak}
recall that $(A_j)_{j=1}^n$ are fully independent
over
$\mathcal{N}$ if for every subset $I\subset\{1,\ldots,n\}$ the algebra
$\mathcal{M}(I)$ generated by $\bigcup_{i\in I} A_i$ is independent
from $\mathcal {M}(I^c)$ over $\mathcal{N}$.
%
\begin{theorem}\label{rosnc} Let $(A_i)$ be fully independent over
$\mathcal{N}$, $1\le p<\infty$, $x_i\in L_p(A_i)$ with $E_\N
(x_i)=0$. Then
%
\begin{eqnarray}
\label{rosnc1} &&\Biggl\llVert\sum_{j=1}^n
x_j\Biggr\rrVert_p \le C \max\Biggl\{\sqrt{p}\Biggl
\llVert\Biggl(\sum_{j=1}^nE_\N
\bigl(x_jx_j^*+x_j^*x_j
\bigr) \Biggr)^{1/2}\Biggr\rrVert_p,\nonumber\\[-8pt]\\[-8pt]
&&\hspace*{174pt} p \Biggl(\sum
_{j=1}^n \|x_j\|_p^p
\Biggr)^{1/p} \Biggr\}.\nonumber
\end{eqnarray}
If moreover, $p\ge2.5$, then
%
\begin{eqnarray}
\label{rosnc2}
&&
\Biggl\llVert\sum_{j=1}^n
x_j\Biggr\rrVert_p \le C'\max\Biggl\{
\sqrt{p}\Biggl\llVert\Biggl(\sum_{j=1}^nE_\N
\bigl(x_jx_j^*+x_j^*x_j
\bigr) \Biggr)^{1/2}\Biggr\rrVert_p,\nonumber\\[-8pt]\\[-8pt]
&&\hspace*{191.4pt} p\bigl\|(x_j)
\bigr\|_{L_p(\ell_\infty)} \Biggr\}.\nonumber
\end{eqnarray}
\end{theorem}

According to \cite{Pi} and \cite{J1}, the norm of $(x_j)$ in
$L_p(\ell_\infty)$ is given by\break $\inf\{\|a\|_{2p}\*\|b\|_{2p}\}$ such that
\[
x_j= ay_jb \qquad\mbox{with } \|y_j
\|_{\infty}\le1.
\]
Clearly, the orders $\sqrt{p}$ and $p$ in the above theorem are optimal
because they are already optimal in commutative probability. Note that
in this version Theorem \ref{rosnc} improves on Corollary \ref{Ros} for
$p$ large enough. The passage from first assertion to the second
follows from an argument in \cite{JX3}. After we put this paper on
\href{http://arxiv.org/}{arXiv.org} and submitted it for publication,
S. Dirksen, being aware of our work, showed us his different proof of
(\ref{rosnc1}) and (\ref {cs1}) with slightly better constants (private
communication). Two months later, J. A. Tropp informed us that he
obtained a particular case (i.e., the random matrix version) of
(\ref{rosnc1}) with several coauthors independently by using a
different method in a later paper \cite{Tr2}. In fact, Rosenthal
inequalities in the noncommutative setting have been successively
explored in \cite{JX1,JX2} and \cite {JX3}. The martingale situation is
completely settled due to the work of \cite{Ra} which shows that for
noncommutative martingales,
\[
\biggl\llVert\sum_j d_j\biggr
\rrVert_p \le Cp \biggl(\biggl\llVert\biggl(\sum
_k E_{k-1}\bigl(d_kd_k^*+d_k^*d_k
\bigr) \biggr)^{1/2}\biggr\rrVert_p + \biggl(\sum
_k \| d_k\|_p^p
\biggr)^{1/p} \biggr),
\]
where $(d_k)$ is a sequence of martingale differences given by
$E_k(x)=E_{\N_k}(x)$ and $d_k=d_k(x)=E_k(x)-E_{k-1}(x)$ for a
filtration $(\N_k)\subset\M$. As observed in \cite{JX2}, the constant
$Cp$ gives the correct order.\eject

Let us return to the situation in compressed sensing. Here we obtain
the following result.
%
\begin{cor}\label{cs}
Let $x_j\in\N$ be positive operator, $\tau$ a normalized trace such
that:\vspace*{6pt}

\textup{(i)} $\frac{1}{m}\sum_{j=1}^m x_j=1$; \textup{(ii)} $\|x_j\|
\le r$.\vspace*{6pt}

\noindent Let $\delta_j$ be independent selectors such that $\ez\delta
_j={k}/{m}$. Then for $p\ge2.5$,
%
\begin{equation}
\label{cs1} \Biggl(\ez\Biggl\llVert\frac{1}{k} \sum
_{j=1}^m \delta_j x_j-1
\Biggr\rrVert_{L_p(\tau)}^p \Biggr)^{1/p} \le C \max
\biggl\{\sqrt{\frac{pr}{k}}, \frac{pr}{k} \biggr\}.
\end{equation}
Moreover, if $\mathrm{tr}$ is a trace on $\mathcal{N}$ such that
\[
\|x\|_{L_{\infty}(\mathrm{tr})} \le\|x\|_{L_p(\mathrm{tr})}
\]
and $r/k=\eps^2$, then, for $t^2\ge2.5C^2e$ and $t\ge2.5Ce\eps$, we have
%
\begin{equation}
\label{cs2}\quad  \operatorname{Prob} \Biggl(\Biggl\llVert\frac{1}{k}\sum
_{j=1}^m \delta_{j}x_j-1
\Biggr\rrVert_{L_{\infty}(\mathrm{tr})}>t\eps\Biggr) \le \operatorname{tr}(1) %
\cases{
e^{-{t^2}/(2C^2e)}, &\quad if $t\eps\le C$,
\cr
e^{-{t}/(2Ce\eps)}, &\quad if $t\eps\ge C$.}\hspace*{-35pt}
\end{equation}
Here $C$ is an absolute constant.
\end{cor}

These results are closely related to the matrix Bernstein inequality
from Tropp's paper \cite{Tr} and operator Bernstein inequality from
\cite{Gr}. Their application to problem in compressed sensing will be
explained in Section \ref{sec4}. Section \ref{sec1} provides the proof
of the Bennett's inequality and its consequences. An application to
large deviation inequalities and how noncommutative Gaussian random
variables may violate the classical equalities are discussed in Section
\ref{sec2}. The improved Rosenthal inequality is proved in Section
\ref{sec3}.

\section{Noncommutative Bennett inequality}\label{sec1}
Let us first recall some background. For a self-adjoint operator $a\in
\mathcal{M}$, we have the spectral decomposition $a=\int t \,dE_t$, where
$E_t$ is the spectral measure of $a$. For any Borel set $A\subset
\mathbb
{R}$, we define $\mu(A)=\tau(E(A))$. Then $\mu$ is a scalar-valued
spectral measure for $a$ and $\mu(\mathbb{R})=1$. By the measurable
functional calculus (see, e.g., \cite{Co}, Section IX.8), there exists
a $*$-homomorphism $\pi\dvtx  L^\infty(\mu)\to\mathcal{M}$ depending on $a$
such that for all $f\in L^\infty(\mu), \pi(f)=f(a)$ and
%
\begin{equation}
\label{spe} \tau\bigl(f(a)\bigr)=\int f(t)\mu(dt).
\end{equation}
In particular, for $f=1_{[t,\infty)}$, we have the exponential
Chebyshev inequality
%
\begin{equation}
\label{chb} \tau\bigl(1_{[t,\infty)}(a)\bigr)=\operatorname{Prob}(a \ge t)\le
e^{-t}\tau\bigl(e^a\bigr).
\end{equation}

Our proof of Bennett's inequality relies on the well-known
Golden--\break Thompson inequality. For the usual trace on $B(H)$ we may refer
to Simon's book \cite{Si}. The fully general case is due to Araki
\cite
{Ar}. A transparent proof for semifinite von Neumann algebras can be
found in Ruskai's paper (\cite{Ru},~Theorem~4).
%
\begin{lemma}[(Golden--Thompson inequality)]\label{lemGT}
Suppose that $a,b$ are self-adjoint operators, bounded above and that
$a+b$ are essentially self-adjoint (i.e., the closure of $a+b$ is
self-adjoint). Then
\[
\tau\bigl(e^{a+b}\bigr)\le\tau\bigl(e^{a/2}e^be^{a/2}
\bigr).
\]
Furthermore, if $\tau(e^a)<\infty$ or $\tau(e^b)<\infty$, then
%
\begin{equation}
\label{eqGT} \tau\bigl(e^{a+b}\bigr)\le\tau
\bigl(e^ae^b\bigr).
\end{equation}
\end{lemma}

Note that if $a,b\in\mathcal{M}$ are self-adjoint, the hypotheses in
Lemma \ref{lemGT} are automatically satisfied. Therefore we have
(\ref{eqGT}). With the help of (\ref{chb}) and (\ref{eqGT}), we
can prove
the noncommutative Bennett inequality following the commutative case
given in \cite{Be}.
\begin{pf*}{Proof of Theorem \ref{Ben}} (\ref{chb}) implies for
$\lambda\ge0$,
%
\begin{equation}
\label{eqlg0} \operatorname{Prob} \Biggl(\sum_{i=1}^n
a_i\ge t \Biggr)\le e^{-\lambda t}\tau\bigl(e^{\lambda\sum_{i=1}^n a_i}
\bigr).
\end{equation}
Since $(a_i)$ are successively independent, we deduce from (\ref
{eqGT}) that
%
\begin{eqnarray}
\label{eqlg} \tau\bigl(e^{\lambda\sum_{i=1}^n a_i}\bigr)&\le&\tau
\bigl(e^{\lambda\sum_{i=1}^{n-1}
a_i}e^{\lambda a_n}\bigr)=\tau\bigl(E_\mathcal{N}
\bigl(e^{\lambda\sum_{i=1}^{n-1}
a_i}e^{\lambda a_n}\bigr)\bigr)
\nonumber\\[-8pt]\\[-8pt]
&=&\tau\bigl(E_\mathcal{N}\bigl(e^{\lambda\sum_{i=1}^{n-1} a_i}\bigr
)E_\mathcal
{N}\bigl(e^{\lambda a_n}\bigr)\bigr).\nonumber
\end{eqnarray}
Expanding, we obtain
\begin{eqnarray*}
E_\mathcal{N}\bigl(e^{\lambda a_n}\bigr)&=& E_\mathcal{N} \Biggl(
\sum_{k=0}^\infty\frac{(\lambda a_n)^k}{k!} \Biggr)=
\sum_{k=0}^\infty\frac
{\lambda^k}{k!}E_\mathcal{N}
\bigl(a_n^k\bigr)
\\
&=&1+\sum_{k=2}^\infty\frac{\lambda^k}{k!}E_\mathcal
{N}\bigl(a_n^2a_n^{k-2}\bigr)\le1+
\sum_{k=2}^\infty\frac{\lambda^k}{k!}M_n^{k-2}
\sigma^2_n
\\
&=&1+\frac{\sigma_n^2}{M_n^2}\bigl(e^{\lambda M_n}-1-\lambda M_n\bigr
)\le
\exp\biggl(\frac{\sigma_n^2}{M_n^2}\bigl(e^{\lambda M_n}-1-\lambda M_n
\bigr) \biggr).
\end{eqnarray*}
Note that the function $f(x):=\exp(x^{-2}(e^{\lambda x}-1-\lambda x))$
is increasing for $x>0$. It follows that
\[
E_\mathcal{N}\bigl(e^{\lambda a_n}\bigr)\le\exp\biggl(
\frac{\sigma_n^2}{C^2}\bigl(e^{\lambda C}-1-\lambda C\bigr) \biggr),
\]
where $C=\sup_{i=1,\ldots, n} M_i$.
Iterating $n-2$ times, we obtain
\[
\tau\bigl(e^{\lambda\sum_{i=1}^n a_i}\bigr)\le\exp\biggl(\frac{\sum
_{i=1}^n\sigma_i^2}{C^2}
\bigl(e^{\lambda C}-1-\lambda C\bigr) \biggr).
\]
This yields
%
\begin{equation}
\label{be} \operatorname{Prob} \Biggl(\sum_{i=1}^n
a_i\ge t \Biggr)\le\exp\biggl(-\lambda t+\frac{\sum_{i=1}^n \sigma_i^2}{C^2}
\bigl(e^{\lambda C}-1-\lambda C\bigr) \biggr).
\end{equation}
By differentiating we find the minimizing value $\lambda=C^{-1} \log
(1+{tC}/({\sum_{i=1}^n\sigma_i^2}))$.
Then (\ref{be}) yields the assertion.
\end{pf*}
\begin{pf*}{Proof of Corollary \ref{Bern}}
Note that $\phi(x)\ge{x^2}/{(2+2x/3)}$ and that $\phi(x)\ge
(x/2)\operatorname{arcsinh}(x/2)$ for $x\ge0$. Then the corollary follows by
relaxing the bound in Bennett's inequality.
\end{pf*}

In the following we use Corollary \ref{Bern} to prove Corollary \ref
{Ros}. Let $a\in\mathcal{M}$ be positive. Recall that
$\operatorname{Prob}(a>t)$ is an analog of the classical distribution
function of $a$. In particular, we may use it to compute the $L_p$ norm
of $a$. Indeed, by the same argument as commutative case, for $p>0$ and
positive $a\in \mathcal{M}$, we have
%
\begin{equation}
\label{pnorm} \|a\|_p^p= p\int_0^\infty
t^{p-1}\operatorname{Prob}(a>t) \,dt.
\end{equation}

Recall that the Gamma function is defined as $\Gamma(p)=\int_0^\infty
e^{-r} r^{p-1} \,dr$, and the incomplete Gamma function is defined as
$\Gamma(\alpha,p)=\int_p^\infty e^{-t}t^{\alpha-1}\,dt$. We need an
elementary estimate for $\Gamma(\alpha,p)$. Note that for $t\ge p\ge
2(\alpha-1)$, we have
\[
\bigl(e^{-t}t^{\alpha-1}\bigr)'=-e^{-t}t^{\alpha-1}
\biggl(1-\frac{\alpha
-1}{t} \biggr)\le-\frac12e^{-t}t^{\alpha-1}.
\]
This gives the following lemma.
%
\begin{lemma}\label{lemgamma}
If $p\ge2\alpha-2$, then $\Gamma(\alpha, p)\le2e^{-p}p^{\alpha-1}$.
\end{lemma}
\begin{pf*}{Proof of Corollary \ref{Ros}}
First note that symmetry and Corollary \ref{Bern} imply
\[
\operatorname{Prob} \Biggl(\Biggl\llvert\sum_{i=1}^n
a_i\Biggr\rrvert\ge t \Biggr)\le2\exp\biggl(-\frac{t^2}{2\sum
_{i=1}^n\sigma_i^2+
(2t/3)\sup_{1\le
i\le n}M_i}
\biggr).
\]
Put $S=\sum_{i=1}^n\sigma_i^2$ and $R=\sup_{i=1,\ldots, n}M_i$. By
(\ref{pnorm}), we have
\begin{eqnarray*}
\Biggl\llVert\sum_{i=1}^n
a_i\Biggr\rrVert_p^p&\le&2 p\int
_0^\infty\exp\biggl(-\frac{t^2}{2S+2tR/3}
\biggr)t^{p-1}\,dt
\\
&=& 2p\int_0^{{3S}/{R}}\exp\biggl(-
\frac{t^2}{2S+2tR/3} \biggr)t^{p-1}\,dt\\
&&{}+2p\int_{{3S}/{R}}^\infty
\exp\biggl(-\frac
{t^2}{2S+2tR/3} \biggr)t^{p-1}\,dt
\\
&=&2p(I+\mathit{II}),
\end{eqnarray*}
where
\[
I=\int_0^{{3S}/{R}}\exp\biggl(-\frac{t^2}{2S+2tR/3}
\biggr)t^{p-1}\,dt
\]
and
\[
\mathit{II}=\int_{{3S}/{R}}^\infty
\exp\biggl(-\frac{t^2}{2S+2tR/3
} \biggr)t^{p-1}\,dt.
\]
We first estimate $I$. Since $t\le{3S}/{R}$, we have
\[
I\le \int_0^{{3S}/{R}}e^{-{t^2}/{(4S)}}t^{p-1}\,dt=
2^{p-1}S^{{p}/2}\int_0^{{9S}/({4R^2})}e^{-r}r^{{p}/2-1}\,dr.
\]
For ${9S}/{(4R^2)}\le p$, we have $I\le2^{p-1}S^{{p}/2}\int
_0^pe^{-r}r^{{p}/2-1}\,dr\le2^{p}S^{{p}/2}p^{{p}/2-1}$.
For ${9S}/{(4R^2)}> p$, we have
\begin{eqnarray*}
I&\le&2^{p-1}S^{{p}/2} \biggl(\int_0^pe^{-r}r^{{p}/2-1}\,dr+
\int_p^{
{9S}/({4R^2})}e^{-r}r^{{p}/2-1}\,dr
\biggr)\\
&\le&2^{p}S^{{p}/2}p^{{p}/2-1} + I_2,
\end{eqnarray*}
where $I_2=2^{p-1}S^{{p}/2}\int_p^\infty e^{-r}r^{{p}/2-1}\,dr$,
and by Lemma \ref{lemgamma}, $I_2\le2^{p}S^{{p}/2}
p^{{p}/2-1}\*e^{-p}$. Hence, we obtain
\[
I\le2^{p+1}S^{{p}/2} p^{{p}/2-1}.
\]
To estimate $\mathit{II}$, since $2S<2tR/3$, we have
\begin{eqnarray*}
\mathit{II}&\le&\int_{{3S}/{R}}^\infty e^{-{3t}/{(4R)}}t^{p-1}\,dt=
\biggl(\frac43R \biggr)^p\int_{9S/
({4R^2})}^\infty e^{-r}r^{p-1}\,dr
\\
&\le&\biggl(\frac43R \biggr)^p \Gamma(p)\le\biggl(\frac43Rp
\biggr)^p.
\end{eqnarray*}
Combining all the inequalities together, we find $\llVert\sum_{i=1}^n
a_i\rrVert_p^p\le2^{p+2}S^{{p}/2} p^{{p}/2}+2(4R/3)^p p^{p+1}$.
Hence, we obtain
\[
\Biggl\llVert\sum_{i=1}^n
a_i\Biggr\rrVert_p\le4\sqrt{Sp}+\frac{4\sqrt{2}}3e^{1/e}Rp
\le4(\sqrt{Sp}+Rp).
\]
\upqed
\end{pf*}
We remark that the constant in the above inequality is explicit and
quite small, which may be good for numerical purpose.

\section{Large deviation principle}\label{sec2}
Bennett's inequality is a large deviation type inequality giving an
upper bound for the tail probability. In the commutative setting lower
bounds have been analyzed intensively in large deviation theory.
Despite the fact that our arguments in the previous section are almost
commutative, lower bounds for noncommutative random variables are very
different. Let us start with Cram\'{e}r's theorem. We consider a
sequence of fully independent and identically distributed (i.i.d.)
$\tau
$-measurable (see, e.g., \cite{FK}) noncommutative random variables
$(a_i)_{i\in I}$.

Let $\Lambda(\lambda)=\log\tau(e^{\lambda a_1})$. Following \cite{DZ}
we define the Fenchel--Legendre transform of $\Lambda(\lambda)$ for
$x\in\mathbb{R}$
%
\begin{equation}
\label{fl} \Lambda^*(x)=\sup_{\lambda\in\mathbb{R}}\bigl[\lambda
x-\Lambda(\lambda)
\bigr].
\end{equation}
If $(a_i)$ is a commutative i.i.d. sequence, then Cram\'er's theorem
(\cite{DZ}, Theorem~2.2.3) says that $(a_i)$ satisfies the large
deviation principle (LDP) with rate function $\Lambda^*$, which implies
\cite{DZ}, Corollary 2.2.19,
%
\begin{equation}
\label{eqldp} \limsup_{n\to\infty}\frac1n\log\operatorname{Prob} \Biggl(
\sum_{i=1}^n a_i \ge nt
\Biggr)=-\inf_{s\ge t}\Lambda^*(s).
\end{equation}
The upper bound remains valid in the noncommutative setting.
%
\begin{prop}
Let $(a_i)_{i\ge1}$ be an i.i.d. sequence in $(\mathcal{M},\tau)$ such
that $\tau(a_i)=0$ for all $i\ge1$. Then for any $t>0$,
\[
\limsup_{n\to\infty}\frac1n\log\operatorname{Prob} \Biggl(\sum
_{i=1}^n a_i \ge nt \Biggr)\le-
\inf_{s\ge t}\Lambda^*(s).
\]
\end{prop}
\begin{pf}
Thanks to the Golden--Thompson inequality, we can follow the proof in
the commutative case in \cite{DZ}. Using (\ref{eqlg0}) and (\ref
{eqlg}), we obtain
\[
\operatorname{Prob} \Biggl(\sum_{i=1}^n
a_i \ge nt \Biggr)\le e^{-\lambda
nt}\prod
_{i=1}^n\tau\bigl(e^{\lambda a_i}\bigr)=
e^{-n(\lambda t-\Lambda(\lambda
))}.
\]
This implies
\[
\frac1n\log\operatorname{Prob} \Biggl(\sum_{i=1}^n
a_i \ge nt \Biggr)\le-\Lambda^*(t)\le-\inf_{s\ge t}
\Lambda^*(s).
\]
\upqed
\end{pf}
%
\begin{rem}
Although we assumed $a_i$'s are in $(\mathcal{M},\tau)$, using
truncation and approximation, we can also prove the previous
proposition for symmetric Gaussians. To be more precise, for
independent symmetric Gaussian random variables $a$ and $b$, let
$a_N=a1_{\{|a|<N\}}$ and $b_N=b1_{\{|b|<N\}}$. Then the monotone
convergence theorem implies that $\tau(e^{a_N})\to\tau(e^{a}),\tau
(e^{b_N})\to\tau(e^{b})$. Since the symmetric Gaussian random variable
is in $\bigcap_{p\ge1}L_p(\M,\tau)$, the triangle inequality implies
$\tau((a_N+b_N)^p)\to\tau((a+b)^p)$. By symmetry, we have
\[
\tau\bigl(e^{a_N+b_N}\bigr)\to\tau\bigl(e^{a+b}\bigr).
\]
\end{rem}

In the following we give two examples which violate the LDP for
noncommutative random variables.
%
\begin{exam}[(Noncommutative semicircular law \cite{VDN})]
\label{semi}
Recall that the \textit{semicircular law} centered at $a\in\mathbb{R}$
and of radius $r>0$ is the distribution $\gamma_{a,r}\dvtx \mathbb
{C}[X]\to
\mathbb{C}$ defined by
\[
\gamma_{a,r}(P)=\frac2{\pi r^2}\int_{a-r}^{a+r}P(t)
\sqrt{r^2-(t-a)^2} \,dt.
\]
Here $\mathbb{C}[X]$ is the algebra of complex polynomials in one variable.

Let us recall that copies of semicircular random variables can be
constructed on the full Fock space; see, for example, \cite{VDN}, Section
2.6. We find a sequence of the so-called free (thus fully
independent) Gaussian random variables $\{s_i\}_{i\in I}$ with the
identical distribution $\gamma_{0,2}$. By rotation invariance of the
free functor, we deduce from \cite{VDN}, Section 3.4, that
%
\begin{equation}
\label{eqinv} \hat{s}_n=\frac1{\sqrt{n}}\sum
_{i=1}^n s_i\sim\gamma_{0,2},
\end{equation}
which means that the distribution of $\hat{s}_n$ is $\gamma_{0,2}$.
Since $\gamma_{0,2}$ is supported in $[-2,2]$, for any $t>0$,
\[
\lim_{n\to\infty}\frac1n\log\operatorname{Prob} \Biggl(\sum
_{i=1}^n s_i\ge nt \Biggr)=
\lim_{n\to\infty}\frac1n\log\operatorname{Prob}(\hat{s}_n\ge\sqrt{n}t)=-
\infty.
\]
On the other hand, by the integral representation of the modified
Bessel function~$I_1$ (\cite{Te}, (9.46)), the moment generating function
of $\gamma_{0,2}$ is given by
\[
M(\lambda)=\frac1{2\pi}\int_{-2}^2
e^{\lambda t}\sqrt{4-t^2} \,dt= \frac{I_1(2\lambda)}{\lambda}.
\]
Using the series representation of $I_1$ (\cite{Te}, (9.28)), we have for
$\lambda>0$,
\[
M(\lambda)=\sum_{n=0}^\infty
\frac{\lambda^{2n}}{(n+1)!n!}\ge\sum_{n=0}^\infty
\frac{\lambda^{2n}}{2(2n)!}=\frac{e^{\lambda
}+e^{-\lambda}}4\ge\frac14e^{\lambda}.
\]
We find $\Lambda(\lambda)=\log M(\lambda)\ge\lambda-\log4$. Since
$\tau(a_1)=0$, by \cite{DZ}, Lemma 2.2.5, for $x\ge0$,
\[
\Lambda^*(x)=\sup_{\lambda\ge0}\bigl[\lambda x-\Lambda(\lambda)\bigr].
\]
Therefore,
\[
\Lambda^*(1)=\sup_{\lambda\ge0}\bigl[\lambda-\Lambda(\lambda)\bigr]\le
\log4<
\infty,
\]
which shows that the sequence $(s_i)$ violates the LDP lower bound in
(\ref{eqldp}). We have proved the following result.
%
\begin{prop}
The semicircular sequence $(s_n)_{n\in\mathbb{N}}$ does not satisfy
LDP~(\ref{eqldp}).
\end{prop}
\end{exam}
The counterexample works in free probability because $s_1$ is bounded.
In order to motivate the next example, we first clarify the
relationship between the logarithmic moment generating function
$\Lambda
$ and the rate function $I$ of the LDP.
%
\begin{rem}\label{ldp2}
Suppose that an i.i.d. sequence $(a_n)$ satisfies the LDP with rate
function $I(x)$ and that $\Lambda(\lambda)$ is well defined. Then the
Fenchel--Legendre transform of $I(x)$ coincides with $\Lambda
(\lambda)$, that is,
\[
I^{*}(\lambda)=\Lambda(\lambda).
\]
Indeed, by H\"{o}lder's inequality $\Lambda(\lambda)$ is convex, and by
Fatou's lemma for $\tau$-measurable operators (\cite{FK}, Theorem 3.5),
$\Lambda(\lambda)$ is lower semicontinuous. Then Cram\'er's theorem and
the duality lemma (\cite{DZ}, Lemma 4.5.8) yield the assertion. In
particular, if $(a_n)$ satisfies the LDP with rate function $I(x)$ and
$\Lambda(\la)$ exists, then $I(x)={x^2}/2$ implies $\Lambda(\lambda
)=I^*(\lambda)={\lambda^2}/2$; that is, the sequence $(a_n)$ follows
standard normal distribution. This means in classical probability
\textit{the distribution of an i.i.d. sequence can be recovered from the rate
function} given by the LDP. The next proposition will show that this is
no longer the case in the noncommutative setting. Therefore, a literal
translation of the LDP is not to be expected in noncommutative probability.
\end{rem}
%
\begin{prop}[(Gaussian family)]\label{ldp3}
Let $\tet\in(0,1)$. There exists an i.i.d. sequence $(\xi_n)_{n\ge1}$
of noncommutative Gaussian random variables with logarithmic moment
generating function $\Lambda_\tet(\lambda)$ such that:
\begin{longlist}[(ii)]
\item[(i)] $(\xi_n)$ satisfies the LDP with rate function
$I_\theta(x)={x^2}/2$;
\item[(ii)] $ \llvert\Lambda_\theta(\lambda)-\frac{\lambda
^2}2-\log
(1-\theta)\rrvert\le\frac\theta{1-\theta}e^{2\la-\lambda^2/2}$.
\end{longlist}
In particular, $I_\theta^*(\la)=\la^2/2\neq\Lambda_\theta(\la)$.
Therefore, the law of $(\xi_n)$ cannot be recovered from the LDP rate function.
\end{prop}
Before going to the proof, we remark that the failure of recovering the
law $\Lambda_\tet(\cdot)$ from rate function $I_\theta(\cdot)$ is
because Cram\'{e}r's theorem is no longer true in the noncommutative
setting. Indeed, since $\Lambda_0(\lambda)=\lambda^2/2$ is the
logarithmic moment generating function of standard normal distribution
and $\Lambda_0^*(x)=x^2/2$, if Cram\'{e}r's theorem were true, we would
have $\Lambda_\tet^*(x)=I_\tet(x)=x^2/2=\Lambda_0^*(x)$. But
$\Lambda_\theta(\lambda)\neq\Lambda_0(\lambda)$ as stated above, this
contradicts the injectivity of Fenchel--Legendre transform. 
%
\begin{pf*}{Proof of Proposition \ref{ldp3}}
For $\theta\in(0,1)$, given a noncommutative standard Gaussian random
variable $g_0$ (with probability density function $e^{-x^2/2}/\sqrt
{2\pi
}$) and a noncommutative semicircular random variable $g_1\sim\gamma
_{0,2}$, there exists a noncommutative random variable $g_\tet$ such that
\[
\tau\bigl(g_\theta^k\bigr)=(1-\theta)\tau
\bigl(g_0^k\bigr)+\theta\tau\bigl(g_1^k
\bigr).
\]
This implies by approximation (see \cite{J2})
\[
\tau\bigl(f(g_\theta)\bigr)=(1-\theta)\tau\bigl(f(g_0)
\bigr)+\theta\tau\bigl(f(g_1)\bigr)
\]
for all measurable function $f$. In particular, for any Borel set
$A\subset\mathbb{R}$,
%
\begin{equation}
\label{ldp1} 
\tau\bigl(1_A(g_\tet)
\bigr)=(1-\theta)\tau\bigl(1_A(g_0)\bigr)+\theta\tau
\bigl(1_A(g_1)\bigr)
\end{equation}
and for all $\lambda\in\mathbb{R}$,
%
\begin{equation}
\label{ldpm} \tau\bigl(e^{\lambda g_\theta}\bigr)=(1-\theta)\tau\bigl
(e^{\lambda
g_0}
\bigr)+\theta\tau\bigl(e^{\lambda g_1}\bigr).
\end{equation}
Moreover, for every real Hilbert space $H$ there exists an algebra $\N
_\tet(H)$, together with a map $u\dvtx H\to\N_\tet(H)$ and a family of
trace preserving automorphisms $\al_o\dvtx \N_\tet(H)\to\N_\tet(H)$ indexed
by the contractions $o$ of $H$ such that
\[
\al_o\bigl(u(h)\bigr) = u\bigl(o(h)\bigr).
\]
We apply this for $H=\ell_2(\nz)$ and define $\xi_i=u(e_i)$ where
$\xi_1$ has the same distribution as $g_\tet$. Using the permutations, we
see that $\xi_i=\al_{(i1)}(\xi_1)$, and hence these variables are
identical distributed. Using the conditional expectations onto $N_\tet
(\ell_2(I))$, $I\subset\nz$, we see that $(\xi_i)$ is a fully
independent sequence. Using\vspace*{1pt} a real unitary which maps $e_1$ to $\frac
{1}{\sqrt{n}}\sum_{i=1}^n e_i$, we deduce that $\xi_1$ and $\frac
{1}{\sqrt{n}}\sum_{i=1}^n \xi_i$ have the same distribution, that is,
%
\begin{equation}
\label{inv2} \frac{1}{\sqrt{n}}\sum_{i=1}^n
\xi_i \stackrel{D} {=} \xi_1 \stackrel{D} {=}
g_\tet;
\end{equation}
see \cite{GM,GM2,CJ} for more details.
Following \cite{DZ}, Section 2.2, we define $S_n=\frac1n\sum_{k=1}^n
\xi_k$ and $\mu_n(A)=\tau(1_A(S_n))$.
By the invariance property (\ref{inv2}), we have $\mu_n(A)=\tau
(1_{\sqrt
{n}A}(\sqrt{n}S_n))=\tau(1_{\sqrt{n}A}(g_\theta))$.
Using (\ref{ldp1}), we find
%
\begin{equation}
\label{ldp4} \mu_n(A)=\tau\bigl(1_{\sqrt{n}A}(g_\theta)
\bigr)=(1-\tet)\tau\bigl(1_{\sqrt
{n}A}(g_0)\bigr)+\tet\tau
\bigl(1_{\sqrt{n}A}(g_1)\bigr).
\end{equation}
%
We aim to establish an LDP for $(\mu_n)$. Let $A$ be a Borel set and
$I_\tet(x)=x^2/2$. Note that the support of the distribution of $g_1$
is $[-2, 2]$. We consider the following two cases:
\begin{longlist}[(2)]
\item[(1)] $0\in\operatorname{cl}(\operatorname{int}(A))$, the closure of interior of $A$.
If there exists an interval $(-\delta,\delta)\subset
\operatorname{cl}(\operatorname{int}(A))$, then $\lim_{n\to\infty}\tau(1_{\sqrt{n}A}(g_1))=1$.
If no
such interval exists, 0 is a boundary point of $\operatorname{cl}(\operatorname{int}(A))$,
then $\lim_{n\to\infty}\tau(1_{\sqrt{n}A}(g_1))=1/2$. In any case,
we have
\[
\lim_{n\to\infty}\frac1n\log\mu_n(A)=0=-\inf_{x\in
\operatorname{int}(A)}I_\tet(x)=-
\inf_{x\in\operatorname{cl}(A)}I_\tet(x).
\]
\item[(2)] $0\notin\operatorname{cl}(\operatorname{int}(A))$. In this case,
$\operatorname{int}(\sqrt{n}A\cap[-2,2])$ will eventually be empty for $n$ large
enough. Then we have $\lim_{n\to\infty}\tau(1_{\sqrt
{n}A}(g_1))=0$. First
we assume $\operatorname{int}(A)\neq\varnothing$ and without loss of
generality, we assume $\operatorname{int}(A)\subset\rz_+$. Let $x=\inf\{
\operatorname{int} A\}$ and $(x,T)$ be an interval contained in $\mathrm{A}$.
Then we have
\[
\int_{\sqrt{n}x}^{\sqrt{n}T}e^{-t^2/2}\,dt\le\int
_{\sqrt
{n}A}e^{-t^2/2}\,dt\le\int_{\sqrt{n}x}^{\infty}e^{-t^2/2}\,dt.
\]
Since $\tau(1_{\sqrt{n}A}(g_0))=\frac1{\sqrt{2\pi}}\int_{\sqrt
{n}A}e^{-t^2/2}\,dt$, straightforward computation shows that
\[
-\frac{x^2}2\le\liminf_{n\to\infty}\frac1n \log\tau
\bigl(1_{\sqrt
{n}A}(g_0)\bigr)\le\limsup_{n\to\infty}\frac1n
\log\tau\bigl(1_{\sqrt{n}A}(g_0)\bigr)\le-\frac{x^2}2.
\]
This fact together with (\ref{ldp4}) yields
%
\begin{equation}
\label{ldp5} -\inf_{x\in\operatorname{int}(A)}I_\tet(x) \le\lim_{n\to\infty}
\frac1n \log\mu_n(A) \le-\inf_{x\in\operatorname{cl}(A)}I_\tet(x).
\end{equation}
Note that if $\operatorname{int}(A)=\varnothing$, (\ref{ldp5}) is trivial.
\end{longlist}
According to \cite{DZ}, (1.2.4), we have shown that $(\mu_n)$ or
$(\xi_n)$ satisfies the LDP with rate function $I_\tet(x)=x^2/2$.
On the other hand, if we put $\Lambda_\theta(\lambda)=\log\tau
(e^{\lambda g_\theta})$ and let $\nu$ denote the probability measure of
$g_1$, then (\ref{ldpm}) implies
\begin{eqnarray*}
\Lambda_\theta(\lambda)&=&\log\biggl((1-\theta)e^{\lambda
^2/2}+\theta
\int_{-2}^2e^{\lambda t}\nu(dt) \biggr)
\\
&=&\log\biggl((1-\theta)e^{\lambda^2/2} \biggl(1+\frac{\theta
e^{-\lambda
^2/2}}{1-\theta}\int
_{-2}^2e^{\lambda t}\nu(dt) \biggr) \biggr)
\\
&\le&\log\biggl((1-\theta)e^{\lambda^2/2} \biggl(1+\frac{\theta
e^{-\lambda
^2/2}e^{2\lambda}}{1-\theta}\int
_{-2}^2\nu(dt) \biggr) \biggr)
\\
&\le&\log(1-\theta)+ \frac{\lambda^2}2+\log\biggl(1+\frac{\theta
}{1-\theta}
e^{2\lambda-\lambda^2/2} \biggr)
\end{eqnarray*}
and similarly,
\[
\Lambda_\theta(\lambda)\ge\log(1-\theta)+\frac{\lambda^2}{2}+\log
\biggl(1+\frac{\theta}{1-\theta} e^{-2\lambda-\lambda^2/2} \biggr).
\]
Combining these two inequalities, we obtain
%
\begin{eqnarray}
\label{diff} \biggl\llvert\Lambda_\theta(\lambda)-\frac{\lambda^2}2-
\log(1-\theta)\biggr\rrvert&\le&\log\biggl(1+\frac{\theta}{1-\theta}
e^{2\lambda
-\lambda^2/2}
\biggr)\nonumber\\[-8pt]\\[-8pt]
&\le&\frac\theta{1-\theta}e^{2\la-\lambda
^2/2},\nonumber
\end{eqnarray}
which implies $\lim_{\la\to\infty}\Lambda_\theta(\lambda
)-{\lambda^2}/2=\log(1-\theta)$. In particular, $\Lambda_\theta(\lambda
)\neq
{\lambda^2}/2$. Since $I_\tet^*(\lambda)=\la^2/2\neq\Lambda_\theta(\la
)$, we have proved that the law $\Lambda_\tet(\cdot)$ of $(\xi_n)$
cannot be recovered from the LDP rate function $I_\tet(\cdot)$.
\end{pf*}


\section{Improved noncommutative Rosenthal's inequality}\label{sec3}
We prove the improved noncommutative Rosenthal inequality and show that
the coefficients cannot be improved in this section. In order to prove
Theorem \ref{rosnc}, we will follow and refine the standard iteration
procedure given in \cite{JX3}, used before by Lust-Piquard \cite{LP}
and Pisier, Gilles and Xu \cite{PX}.
\begin{pf*}{Proof of Theorem \ref{rosnc}} Instead of proving (\ref
{rosnc1}) directly, we prove the following equivalent inequality:
%
\begin{eqnarray}
\label{rova}\qquad
\Biggl\llVert\sum_{j=1}^n x_j\Biggr\rrVert_p &\le& D_p
\max\Biggl\{ \sqrt{p}\Biggl\llVert\Biggl(\sum_{j=1}^nE_\N
\bigl(x_j^*x_j\bigr) \Biggr)^{1/2}\Biggr\rrVert_p,
\nonumber\\[-8pt]\\[-8pt]
&&\hspace*{42.6pt}\sqrt{p}\Biggl\llVert\Biggl(\sum_{j=1}^nE_\N\bigl(x_jx_j^*\bigr)
\Biggr)^{1/2}\Biggr\rrVert_p, p \Biggl(\sum_{j=1}^n \|x_j\|_p^p
\Biggr)^{1/p} \Biggr\},\nonumber
\end{eqnarray}
and we assume at the moment that $D_p$ is the best constant which may
depend on the range of $p$. By \cite{JX3}, Theorem 2.1, (\ref{rova}) is
true for $1\le p \le4$. This is the starting point of our iteration
argument. Assume $p>2$. We only need to show ``$p\Rightarrow2p$.'' Let
$x_i\in L_{2p}(\mathcal{M},\tau)$. Write the conditional expectation
operator $E=E_\mathcal{N}$ in the following proof. Put
\[
A=\sqrt{2p}\Biggl\llVert\Biggl(\sum_{i=1}^nE
\bigl(x_i^*x_i\bigr) \Biggr)^{1/2}\Biggr
\rrVert_{2p} \quad\mbox{and}\quad B= 2p \Biggl(\sum_{i=1}^n
\|x_i\|_{2p}^{2p} \Biggr)^{1/{(2p)}}.
\]
Using \cite{JX1}, Lemma 1.2, and the noncommutative Khintchine
inequality in \cite{Pi} with the right order of best constant, we have
\[
\Biggl\llVert\sum_{i=1}^n
x_i\Biggr\rrVert_{2p}\le2\mathbb{E}\Biggl\llVert\sum
_{i=1}^n\eps_i
x_i\Biggr\rrVert_{2p}\le c\sqrt{p}\max\Biggl\{\Biggl
\llVert\sum_{i=1}^nx_i^*x_i
\Biggr\rrVert^{1/2}_{p},\Biggl\llVert\sum
_{i=1}^nx_ix_i^*\Biggr
\rrVert^{1/2}_{p} \Biggr\},\vadjust{\goodbreak}
\]
where $(\eps_i)$ is a sequence of Rademacher random variables, and
$\ez
$ denotes the corresponding expectation.
Let $y_i=x_i^*x_i-E(x_i^*x_i)$. Then
\[
\Biggl\llVert\sum_{i=1}^nx_i^*x_i
\Biggr\rrVert_{p}\le2\max\Biggl\{\Biggl\llVert\sum
_{i=1}^n y_i\Biggr
\rrVert_p,\Biggl\llVert\sum_{i=1}^n
E\bigl(x_i^*x_i\bigr)\Biggr\rrVert_p
\Biggr\}.
\]
Applying the induction hypothesis, we obtain
\[
\Biggl\llVert\sum_{i=1}^n
y_i\Biggr\rrVert_p\le D_p\max\Biggl\{
\sqrt{p}\Biggl\llVert\Biggl(\sum_{i=1}^nE
\bigl(y_i^2\bigr) \Biggr)^{1/2}\Biggr
\rrVert_{p}, p \Biggl(\sum_{i=1}^n
\|y_i\|_{p}^{p} \Biggr)^{1/{p}} \Biggr
\}.
\]
Note that
\[
E\bigl(y_i^2\bigr)= E\bigl(|x_i|^4
\bigr)-\bigl(E\bigl(|x_i|^2\bigr)\bigr)^2\le
E\bigl(|x_i|^4\bigr).
\]
By \cite{JX1}, Lemma 5.2, we obtain
\begin{eqnarray*}
\Biggl\llVert\sum_{i=1}^nE
\bigl(|x_i|^4\bigr)\Biggr\rrVert_{{p}/2}&\le&
\Biggl\llVert\sum_{i=1}^n E
\bigl(|x_i|^2\bigr)\Biggr\rrVert_p^{({p-2})/({p-1})}
\Biggl(\sum_{i=1}^n \|x_i
\|_{2p}^{2p} \Biggr)^{1/({p-1})}
\\
&=& \bigl({A^2}/{2p} \bigr)^{({p-2})/({p-1})} ({B}/{2p} )^{{2p}/({p-1})}
\\
&=& A^{(2p-4)/(p-1)}B^{{2p}/(p-1)}(2p)^{-({3p-2})/({p-1})}.
\end{eqnarray*}
On the other hand, since $E$ is a contraction on $L_p(\mathcal{M},\tau
)$, we have
\begin{eqnarray*}
\Biggl(\sum_{i=1}^n\|y_i
\|_{p}^{p} \Biggr)^{1/{p}}&=& \Biggl(\sum
_{i=1}^n\bigl\|x_i^*x_i-E
\bigl(x_i^*x_i\bigr)\bigr\|_{p}^{p}
\Biggr)^{1/{p}}
\\
&\le& 2 \Biggl(\sum_{i=1}^n
\bigl\|x_i^*x_i\bigr\|_{p}^{p}
\Biggr)^{1/{p}} =2 \Biggl(\sum_{i=1}^n
\|x_i\|_{2p}^{2p} \Biggr)^{1/{p}}=
\frac
{B^2}{2p^2}.
\end{eqnarray*}
This gives
\begin{eqnarray*}
\Biggl\llVert\sum_{i=1}^n
y_i\Biggr\rrVert_p&\le& D_p\max\biggl\{\sqrt
{p}A^{(p-2)/(p-1)}B^{p/(p-1)}(2p)^{-({3p-2})/(2p-2)},\frac
{pB^2}{2p^2}\biggr\}
\\
&\le& D_p\max\biggl\{
2^{-({3p-2})/(2p-2)}A^{(p-2)/(p-1)}B^{p/(p-1)}p^{-1-1/(2p-2)},
{\frac{B^2}{2p}} \biggr\}.
\end{eqnarray*}
Hence, we find
%
\begin{eqnarray}
\label{rosp1}\qquad
\Biggl\llVert\sum_{i=1}^nx_i^*x_i
\Biggr\rrVert_{p}&\le&\max\biggl\{
2^{-p/(2p-2)}{D_p}A^{(p-2)/(p-1)}B^{p/(p-1)}p^{-1-1/(2p-2)},\nonumber\\[-8pt]\\[-8pt]
&&\hspace*{198.2pt}
\frac
{D_pB^2}{p},\frac{A^2}p \biggr\}.\nonumber
\end{eqnarray}
Young's inequality for products implies
%
\begin{equation}
\label{rosp2} A^{(p-2)/(2p-2)}B^{{p}/(2p-2)}\le\frac{(p-2)A}{2p-2}+
\frac{pB}{2p-2} \le\max\{A, B\}.
\end{equation}
Note that $2^{-p/(4p-4)}\le2^{-1/4}$ and $p^{-1/(4p-4)}\le1$. Equations (\ref
{rosp1}) and (\ref{rosp2}) yield
\begin{eqnarray*}
&&
\sqrt{p}\Biggl\llVert\sum_{i=1}^nx_i^*x_i
\Biggr\rrVert_{p}^{1/2} \\
&&\qquad\le\max\bigl\{ 2^{-1/4}
\sqrt{D_p}\max\{A, B\},\sqrt{D_p} B, A\bigr\}\le\sqrt
{D_p}\max\{A, B\}
\\
&&\qquad=\sqrt{D_p}\max\Biggl\{\sqrt{2p}\Biggl\llVert\Biggl(\sum
_{i=1}^nE\bigl(x_i^*x_i
\bigr) \Biggr)^{1/2}\Biggr\rrVert_{2p},2p \Biggl(\sum
_{i=1}^n\|x_i\|_{2p}^{2p}
\Biggr)^{1/({2p})} \Biggr\}.
\end{eqnarray*}
Applying the same argument to $x_ix_i^*$, we obtain
\begin{eqnarray*}
&&
\sqrt{p}\Biggl\llVert\sum_{i=1}^nx_ix_i^*
\Biggr\rrVert_{p}^{1/2} \\
&&\qquad\le\sqrt{D_p}\max\Biggl
\{\sqrt{2p}\Biggl\llVert\Biggl(\sum_{i=1}^nE
\bigl(x_ix_i^*\bigr) \Biggr)^{1/2}\Biggr
\rrVert_{2p},2p \Biggl(\sum_{i=1}^n
\|x_i\|_{2p}^{2p} \Biggr)^{1/({2p})} \Biggr
\}.
\end{eqnarray*}
Hence, (\ref{rova}) is true for $2p$ with constant $c\sqrt{D_p}$. It
follows that
\[
D_{2p} \le c\sqrt{D_p},
\]
and thus $D_p\le c^2$ which is independent of $p$. Therefore, the
iteration argument is complete, and we have proved the first assertion.
As mentioned in the \hyperref[sec0]{Introduction} of this paper, the interpolation
argument from \cite{JX3}, Section 4, shows that the first assertion
can be improved to the second assertion with a singularity as $p$ tends to~2.
Thus for $p\ge2.5$ the assertion holds with an absolute constant.
\end{pf*}
%
\begin{rem}
The improved Rosenthal inequality allows us to extend
Lust-Piquard's noncommutative Khintchine inequality \cite{LPP,LP} in a
twisted setting. We refer to \cite{Val} for unexplained notion on the
Gaussian measure space construction. The starting point is a discrete
group acting on a real Hilbert space $H$. This means we fix an isometry
$b\dvtx H\to L_2(\Om,\Si,\mu)$ such that $b$ is linear, and $b(h)$ is a
centered Gaussian random variable with variance $\|h\|^2$. For example,
for $H=L_2(0,\infty)$ and $B_t=b(1_{[0,t]})$ we recover a well-known
method to construct Brownian motion. We may assume that $\Si$ is the
minimal sigma algebra generated by the random variables $b(H)$. Then
the\vadjust{\goodbreak} action of $G$ extends to a family of measure preserving
automorphism $\al\dvtx G\to\operatorname{Aut}(L_{\infty}(\Om,\Si,\mu))$ such that
\[
\al_g\bigl(b(h)\bigr)= b(g.h).
\]
This allows us to form the crossed product $M=L_{\infty}(\Si)\rtimes
G$. The crossed product is spanned by random variables of the form
\[
x =\sum_g f_g\la(g).
\]
Here $\la(g)$ refers to the regular representation of group. The
algebraic structure is determined by $\la(g)f\la(g^{-1})= \al_{g}(f)$.
The twisted Gaussian random variables are of the form
\[
B =\sum_g b(h_g)\la(g),\qquad
h_g\in H.
\]
In order to formulate the Khintchine inequality, we have to recall that
there exists trace preserving conditional expectation $E\dvtx M\to L(G)$.
Here $L(G)$ is the von Neumann subalgebra generated by the image $\la
(G)$ and the trace is given by
\[
\tau\biggl(\sum_g f_g \la(g)
\biggr) =\int f_1\,d\mu.
\]
Then we can deduce from Theorem \ref{rosnc} that for $p\ge2$,
%
\begin{equation}
\label{kkk} \|B\|_p \le c\sqrt{p} \bigl\|E\bigl(B^*B+BB^*
\bigr)^{1/2}\bigr\|_p.
\end{equation}
Moreover, the span of the generalized Gaussian random variables is
complemented, and the inequality remains true with additional vector
valued coefficients. This is a key fact in proving noncommutative Riesz
transforms. To illustrate (\ref{kkk}) let us assume that the action is
trivial. Let $(e_k)$ be a basis and
\[
B=\sum_{k,g} a(k,g) b(e_k)\ten\la(g) =
\sum_k b(e_k)\ten a_k.
\]
Then we find
\[
E\bigl(BB^*\bigr) =\sum_{k} a_ka_k^*,\qquad E\bigl(B^*B\bigr) =\sum_k a_k^*a_k.
\]
Thus the right-hand side gives exactly the square function we expect
for Gaussian variables. However, with nontrivial additional group
action $BB^*$ and $B^*B$ look quite different, and the group action
interferes significantly.
\end{rem}

Using (\ref{rosnc2}), we can prove Corollary \ref{cs} which will play a
central role in the application to compressed sensing in the next section.
\begin{pf*}{Proof of Corollary \ref{cs}}
By Jensen's inequality, we have
\begin{eqnarray*}
&& \Biggl(\ez_\delta\Biggl\llVert\frac1k\sum
_{i=1}^m\delta_ix_i-1
\Biggr\rrVert^p_{L_p(\N,\tau)} \Biggr)^{1/p}
\\
&&\qquad= \Biggl(\ez_\delta\Biggl\llVert\frac1k\sum
_{i=1}^m\delta_ix_i-
\frac1k\mathbb{E}_{\delta'}\Biggl(\sum_{i=1}^m
\delta_i'x_i\Biggr)\Biggr
\rrVert^p_{L_p(\N,\tau
)} \Biggr)^{1/p}
\\
&&\qquad\le \Biggl(\ez_\delta\Biggl(\ez_{\delta'}\Biggl\llVert\frac1k
\sum_{i=1}^m\bigl(\delta_i-
\delta_i'\bigr)x_i\Biggr
\rrVert_{L_p(\N,\tau)} \Biggr)^p \Biggr)^{1/p}
\\
&&\qquad\le \Biggl(\ez_{\delta, \delta'}\Biggl\llVert\frac1k\sum
_{i=1}^m\bigl(\delta_i-
\delta_i'\bigr)x_i\Biggr
\rrVert^p_{L_p(\N,\tau)} \Biggr)^{1/p},
\end{eqnarray*}
where $(\delta_i')$ is a sequence of independent selectors with the
same distribution as~$\delta_i$'s. In order to apply Theorem \ref
{rosnc}, it is crucial to choose appropriate probability space. Let
$(\Omega, \mathcal{F}, \mathbb{P})$ be the probability space generated
by $(\delta_i, \delta_i')$. We consider the noncommutative probability
space as the algebra $\mathcal{M}=L_\infty(\mathbb{P})\otimes
\mathcal
{N}$. Then we have a normalized trace $\tilde\tau=\mathbb{E}\otimes
\tau
$ on $\mathcal{M}$. We identify $\mathbb{E}$ as the conditional
expectation $\mathbb{E}\dvtx  \mathcal{M}\to\mathcal{N}$. Clearly,
$((\delta_i-\delta_i') x_i)_{i=1}^n$ are fully independent over
$\mathcal{N}$.
Note that
\[
\mathbb{E}\bigl(\delta_i-\delta_i'
\bigr)^2=\frac{2k}m \biggl(1-\frac
{k}m \biggr)\le
\frac{2k}{m} \quad\mbox{and}\quad \sup_{i=1,\ldots,m}\bigl|\delta_i-
\delta_i'\bigr|\le1.
\]
Since $x_i$ is positive, $x_i^*x_i=x_i^2$. Using (\ref{rosnc2}), we obtain
\begin{eqnarray*}
\Biggl(\mathbb{E}\Biggl\llVert\sum_{i=1}^m
\bigl(\delta_i-\delta_i'
\bigr)x_i\Biggr\rrVert^p_{L_p(\mathcal{N},\tau)}
\Biggr)^{1/p}
&=&\Biggl\llVert\sum_{i=1}^m
\bigl(\delta_i-\delta_i'
\bigr)x_i\Biggr\rrVert_{L_p(\mathcal{M},\tilde{\tau})}
\\
&\le& C\max\Biggl\{\sqrt{p}\Biggl\llVert\sum_{i=1}^m
\mathbb{E}\bigl(\bigl(\delta_i-\delta_i'
\bigr)^2x_i^2\bigr)\Biggr
\rrVert_{L_{p/2}(\mathcal{N},\tau)}^{1/2},\\
&&\hspace*{58.3pt} p\Bigl\|\sup_{i=1,\ldots,m}\bigl|
\delta_i-\delta_i'\bigr|x_i
\Bigr\|_{L_p(\mathcal{M},\tilde
{\tau
})} \Biggr\}.
\end{eqnarray*}
Since $\tau(1)=1$ and $x_i\le r$, we obtain $\||\delta_i-\delta_i'|x_i\|
_{L_p(\mathcal{M},\tilde{\tau}; \ell_\infty)}\le r$, and
\[
\Biggl\llVert\sum_{i=1}^m \mathbb{E}
\bigl(\delta_i-\delta_i'
\bigr)^2x_i^2\Biggr\rrVert_{L_{p/2}(\mathcal{N},\tau)}
\le2kr\Biggl\llVert\frac1m\sum_{i=1}^m
x_i\Biggr\rrVert_{L_{p/2}(\mathcal{N},\tau)}=2kr.
\]
Therefore, we find
\[
\Biggl(\ez\Biggl\llVert\frac1k\sum_{i=1}^m
\bigl(\delta_i-\delta_i'
\bigr)x_i\Biggr\rrVert^p_{L_p(\N,\tau)}
\Biggr)^{1/p}\le C\max\biggl\{\sqrt{\frac
{2pr}{k}},
\frac{pr}{k} \biggr\}.
\]
We have completed the proof of (\ref{cs1}) with constant $\sqrt{2}C$.
For the ``moreover'' part, we use the additional norm assumption and obtain
\[
\Biggl\llVert\frac1k\sum_{i=1}^m
\delta_ix_i-1\Biggr\rrVert_{L_\infty(\mathrm{tr})}\le\Biggl
\llVert\frac1k\sum_{i=1}^m
\delta_ix_i-1\Biggr\rrVert_{L_p(\mathrm{tr})}.
\]
Then by Chebyshev's inequality and (\ref{cs1}) for trace $\tau
(x)={\operatorname{tr}(x)}/{\operatorname{tr}(1)}$, we have
\begin{eqnarray*}
\mathbb{P} \Biggl(\Biggl\llVert\frac1k\sum_{i=1}^m
\delta_ix_i-1\Biggr\rrVert_{L_\infty(\mathrm{tr})}\ge t\eps
\Biggr)&\le&(t\eps)^{-p}\ez\Biggl\llVert\frac1k\sum
_{i=1}^m\bigl(\delta_i-
\delta_i'\bigr)x_i\Biggr
\rrVert^p_{L_p(\mathrm{tr})}
\\
&\le& \operatorname{tr}(1)\max\biggl\{\sqrt{\frac{C^2pr}{kt^2\eps^2}},\frac
{Cpr}{kt\eps
} \biggr
\}^p.
\end{eqnarray*}
Let us first assume $t\eps\le C$. Optimize the first term in $p$ and find
$p={t^2\eps^2k}/({C^2re})$.
Recall that $k=r\eps^{-2}$. Then the first term becomes
$e^{-t^2/(2C^2e)}$. Using $t\eps\le C$, this choice of $p$ gives an
upper bound of $e^{-t^2/(C^2e)}$ for the second term. Now assume $t\eps
\ge C$. The optimal choice for the second term is obtained for
$p={kt\eps}/({Cre})$.
Then the second term becomes $e^{-t/(Ce\eps)}$ and, thanks to $t\eps
\ge
C$, the first term is less than $e^{-t/(2Ce\eps)}$. The additional
assumption on $t$ guarantees that $p\ge2.5$ in both cases. Therefore,
\[
\mathbb{P} \Biggl(\Biggl\llVert\frac1k\sum_{i=1}^m
\delta_ix_i-1\Biggr\rrVert_{L_\infty(\mathrm{tr})}\ge t\eps
\Biggr)\le \operatorname{tr}(1) %
\cases{ e^{-{t^2}/({2C^2e})}, &\quad if $t\eps\le C$,
\cr
e^{-{t}/({2Ce\eps})}, &\quad if $t\eps\ge C$.} %
\]
The constant $C$ is the same as the constant in the first assertion.
\end{pf*}
%
\begin{rem}
In this context it is useful to compare our different generalizations
of Rosenthal's inequality. We
observe that with Corollary \ref{Ros}, we can only obtain
\[
\Biggl(\ez\Biggl\llVert\frac1k\sum_{i=1}^m
\delta_ix_i-1\Biggr\rrVert^p_{L_p(\tau
)}
\Biggr)^{1/p}\le C \biggl(\sqrt{\frac{pr^2}{k}}+\frac{pr}{k}
\biggr),
\]
and with inequality (\ref{rosnc1}) we obtain
\[
\Biggl(\ez\Biggl\llVert\frac1k\sum_{i=1}^m
\delta_ix_i-1\Biggr\rrVert^p_{L_p(\tau
)}
\Biggr)^{1/p}\le C \biggl(\sqrt{\frac{pr}{k}}+\frac
{pr}{k^{1-1/p}}
\biggr).
\]
Both estimates are worse than inequality (\ref{cs1}).
\end{rem}

The following two examples are meant to justify the optimality of
$\sqrt{p}$ and $p$. We refer the reader to \cite{PU} for a more detailed
discussion on this topic in the framework of classical probability. We
will use the standard notation for comparing orders of functions as
$p\to\infty$. Recall that $f(p)=O(g(p))$ if there exists a constant
$C$ such that $f(p)\le C g(p)$ asymptotically, $f(p)=\Omega(g(p))$ if
there exists a constant $c$ such that $f(p)\ge cg(p)$ asymptotically,
$f(p)=\Theta(g(p))$ if there exist constants $c$ and $C$ such that
$cg(p)\le f(p)\le Cg(p)$ asymptotically, and $f(p)\sim g(p)$ if $\lim
_{p\to\infty}{f(p)}/{g(p)}=1$.
%
\begin{exam}[(The optimality of $\sqrt{p}$ in Theorem \ref{rosnc})]
\label{exrp}
Let us assume that
%
\begin{equation}
\label{AB} \Biggl\llVert\sum_{i=1}^n
x_i\Biggr\rrVert_p\le A(p) \Biggl(\sum
_{i=1}^n\|x_i\|^2
\Biggr)^{1/2}+B(p) \Biggl(\sum_{i=1}^n
\|x_i\|^p \Biggr)^{1/p}
\end{equation}
for some functions $A(p)$ and $B(p)$. We use $x_i=g_i$. Here $(g_i)$ is
a sequence of i.i.d. normal random variables with mean 0 and variance
$1$. We know $\ez|g_1|^p=\frac{2^{p/2}}{\sqrt{\pi}}\Gamma(\frac
{p+1}2)$. By Stirling's formula, we obtain for large $p$,
\[
\|g_1\|_p\sim\sqrt{\frac{p}e}.
\]
This yields that there exist absolute constants $c$ and $C$ such that
$c\sqrt{p}\le\|g_1\|\le C\sqrt{p}$ for all $p\ge2$. Hence, we obtain
\[
c\sqrt{p}\le\|g_1\|_p=\Biggl\llVert\frac1{\sqrt{n}}
\sum_{i=1}^ng_i\Biggr
\rrVert_p\le A(p)+ CB(p)\sqrt{p} {n^{1/p-1/2}}.
\]
Sending $n\to\infty$, we have
\[
A(p)\ge c\sqrt{p} \qquad\mbox{for } p>2.
\]
This shows that one cannot reduce the order of $A(p)$, even at the
expense of increasing the order of $B(p)$.
\end{exam}
%
\begin{exam}[(The optimality of $p$ in Theorem \ref{rosnc})]
\label{exmp}
Following Corollary~\ref{cs}, we do a random selector on $\Om=\{1\}$,
that is, $x_i=1$ and $\ez\delta_i=\la=k/m$, and then we shall assume that
\[
\Biggl(\ez\Biggl\llvert\frac{1}{k} \sum_{i=1}^m
\delta_i-1\Biggr\rrvert^p \Biggr)^{1/p} \le C
\sqrt{\frac{p}{k}}+ \frac{f(p)}{k}
\]
for some function $f(p)$. Here we choose $m=p$ and $k=ap$ for some very
small~$a$. Then we find that for every $1\le j\le m$,
\[
\biggl\llvert\frac{j}k-1\biggr\rrvert\pmatrix{m \cr j}^{{1}/{m}}
\la^{j/m} (1-\la)^{1-j/m} \le C \sqrt{\frac{m}k}+
\frac{f(m)}k.
\]
Let us first fix $j=\lceil\gamma m\rceil$ and assume that $\gamma\ge
1/4$ and $1/2^m<a\le1/8$. This gives $\frac{j}{k}\ge\frac{\gamma
}{a}\ge\frac{1}{4a}\ge2$ and hence
\[
\biggl\llvert\frac{j}k-1\biggr\rrvert\ge\frac{1}{8a}.
\]
Note that $1\le{m \choose j}^{1/m}\le2$ so that we cannot expect any
help here. Thus we find
\[
\frac{1}{16} a^{\gamma-1} (1-a)^{1-\gamma} \le\frac{1}{8}
a^{\gamma
-1+1/m} (1-a)^{1-\gamma} \le C a^{-1/2}+
\frac{f(p)}{ap}.
\]
Let us now fix $\gamma=1/4$ and choose $a$ such that
\[
2C a^{-1/2} \le\frac{1}{16} \biggl(\frac{1-a}{a}
\biggr)^{3/4}
\]
or equivalently,
\[
32C a^{1/4} \le(1-a)^{3/4}.
\]
However, $a\le1/8$ implies $1-a\ge7/8$. Thus
\[
a \le\biggl(\frac{7}{8} \biggr)^3 \frac{1}{(32C)^4}
\]
will do. Then we find
\[
\biggl(a^{1/4} \frac{(7/8)^{3/4}}{32} \biggr)p \le f(p).
\]
Choose $a=({7}/{8})^3/{(32C)^4}$. Then we have
%
\begin{equation}
\label{opp1} \frac{c_0}C p \le f(p)
\end{equation}
for an absolute constant $c_0=(7/8)^{3/2}/32^2$. This shows that one
cannot reduce the order of $f(p)$, as long as we keep $A(p)\le C\sqrt
{p}$ in (\ref{AB}).
\end{exam}
%
\begin{rem}\label{oppr}
In fact, Example \ref{exmp} provides more information. Instead of
fixing $\gamma$, by sending $\gamma\to0$ and choosing $a\le\gamma/2$
appropriately, we can find a different behavior. Indeed, then we have
$|j/k-1|\ge\gamma/(2a)$ and
\[
\frac{\gamma}{4}a^{\gamma-1} (1-a)^{1-\gamma}\le Ca^{-1/2}+
\frac
{f(p)}{ap},
\]
and since $a< \gamma$ and $(1-\gamma)^{1-\gamma}\ge e^{-1}$, we need
$ 8eCa^{-1/2} \le\gamma a^{\gamma-1}$
or
\[
a^{1/2-\gamma} \le\frac{\gamma}{8eC}.
\]
Note that $( \frac{\gamma}{8eC})^{{2}/({1-2\gamma})}\le{\gamma}/2$ for
$\gamma\le1$. Hence with
\[
a\le\biggl(\frac{\gamma}{8eC} \biggr)^{{2}/({1-2\gamma})},
\]
we have
\[
\frac{\gamma a^\gamma}{8e} p \le f(p).
\]
Put $a=(\frac{\gamma}{8eC})^{{2}/({1-2\gamma})}$. Then we obtain
\[
\biggl(\frac{\gamma}{8eC} \biggr)^{1/({1-2\gamma})}Cp\le f(p).
\]
Optimizing the left-hand side in $\gamma$, we obtain $2\gamma\log
(8e^2C)-2\gamma\log(\gamma)=1$ and
\[
\biggl(16eC\log\frac{8e^2C}\gamma\biggr)^{-1-1/{\log
({8eC}/\gamma)
}}Cp\le f(p).
\]
Since $\gamma\log\gamma\to0$ as $\gamma\to0$, we choose
\[
\gamma=\frac{1}{2\log(8e^2C)}.
\]
In order to obtain a lower bound for $f(p)$, we need to assume $8C\ge1$
so that $\gamma\le1/4$. This yields for $C\ge1.5$,
%
\begin{equation}
\label{opp2} f(p)\ge\frac{1}{32\sqrt{2}e^{3/2+2/e}\log(8e^2C)}p\ge\frac
{p}{c_1\log C}
\end{equation}
for some absolute constant $c_1$. Compare (\ref{opp2}) with (\ref
{opp1}). Estimate (\ref{opp2}) is better for large $C$. Let us now fix
$p$ and put $C=p^{\alpha}$. Example \ref{exrp} shows that $\alpha$ has
to be nonnegative. (\ref{opp2}) implies that for $\alpha>0$,
\[
f(p)\ge\frac{p}{c_1\alpha\log p}.
\]
In particular, for $C={\sqrt{p}}/{\log p}$, we obtain $f(p)\ge
{2c_1^{-1}p}/\log p$, which recovers the best constants obtained in
\cite{JSZ}.
\end{rem}
Example \ref{exrp} and Remark \ref{oppr} yield the following
result.\eject
%
\begin{theorem}
Under the hypotheses of Theorem \ref{rosnc}, assume that
\[
\Biggl\llVert\sum_{i=1}^n
x_i\Biggr\rrVert_p\le A(p)\Biggl\llVert\Biggl(\sum
_{j=1}^nE_\N
\bigl(x_jx_j^*+x_j^*x_j
\bigr) \Biggr)^{1/2}\Biggr\rrVert_p + B(p) \Biggl(\sum
_{j=1}^n \|x_j
\|_p^p \Biggr)^{1/p}
\]
for some functions $A(p)$ and $B(p)$. Then we have:
\begin{longlist}[(ii)]
\item[(i)] The best possible order of the lower bound for $A(p)$
is $\sqrt{p}$, which cannot be improved, even if the order of $B(p)$ is
increased.
\item[(ii)] If $\Omega(p/\log p)=A(p)=O(p^\beta)$ where
$\beta
\ge1$, then the best possible order of $B(p)$ is $p/\log p$.
\end{longlist}
\end{theorem}

The point here is that the random selector model attains the worst
case in the noncommutative Rosenthal inequality. In the commutative
case, (i)~was proved by Pinelis and Utev in \cite{PU}. Later, Pinelis
proved much stronger results which give different combinations of
best constants in the martingale version of Rosenthal inequality
in the context of Banach spaces. We refer the interested reader to
\cite{Pin} for more details. We thank Pinelis for pointing this
out to us.


\section{Illustration in compressed sensing}\label{sec4}
At the time of this writing there is a large body of work relating
tools originating from noncommutative probability to estimates from
compressed sensing; see \cite{Tr,Tr2} for more details. Since our
improvement of the Rosenthal inequality was motivated by problems in
compressed sensing, we want to describe this relation toward compressed
sensing. Let us briefly recall the background here following \cite
{RV,CRT,CT}. We want to reconstruct an unknown signal $f\in\mathbb{C}^n$
from linear measurements $\Phi f\in\mathbb{C}^k$, where $\Phi$ is some
known $k\times n$ matrix called the measurement matrix. The
reconstruction problem is stated as
%
\begin{equation}
\label{eqop0} \min\bigl\|f^*\bigr\|_0 \quad\mbox{subject to}\quad \Phi
f^*=\Phi f,
\end{equation}
where $\|f\|_0=|{\operatorname{supp} f}|$ is the number of nonzero element of
$f$. Since this problem is computationally expensive, we consider its
convex relaxation instead.
%
\begin{equation}
\label{eqop1} \min\bigl\|f^*\bigr\|_1 \quad\mbox{subject to}\quad \Phi
f^*=\Phi f,
\end{equation}
where $\|f\|_p=(\sum_{j=1}^n|f_j|^p)^{1/p}$ denotes $\ell_p$ norm
throughout this section. Exact reconstruction means that the solutions
to (\ref{eqop0}) and (\ref{eqop1}) are both equal to $f$. $f$~is
assumed to be $s$-sparse, that is, $|{\operatorname{supp} f}|\le s$.
We refer to
\cite{CRT,RV} for why (\ref{eqop1}) is a good substitute of (\ref
{eqop0}). However, the restricted isometry property (RIP) on $\Phi$ is
an extremely important tool for exact reconstruction due to Candes and
Tao \cite{CT2}; see also \cite{CRTV}. Let $\Phi_T$ denote the
$k\times
|T|$ matrix consisting of the columns of $\Phi$ indexed by $T$. The RIP
constant $\Delta_s$ is defined to be the smallest positive number such
that the inequality
\[
C(1-\Delta_s)\|x\|_2^2\le\|
\Phi_Tx\|_2^2\le C(1+\Delta_s)
\|x\|_2^2
\]
holds for some number $C>0$ and for all $x\in\ell_2$ and all subsets
$T\subset\{1,\ldots,n\}$ of size $|T|\le s$. Candes and Tao proved the
following theorem \cite{CT2,CRTV}:
%
\begin{theorem}
Let $f$ be an $s$-sparse signal and $\Phi$ be a measurement matrix
whose RIP constant satisfies
\[
\Delta_{3s}+3\Delta_{4s}\le2.
\]
Then $f$ can be recovered exactly.
\end{theorem}
Since $\Delta_s$ is nondecreasing in $s$, in order to verify RIP, it
suffices to show that
\[
\Delta_{4s}\le\tfrac12
\]
or simply $\Delta_{s}\le\frac12$ by adjusting constant if necessary.
In this section, we apply Corollary \ref{cs} to study the problem of
reconstruction from Fourier measurements. Two cases will be considered.
In the first case, we fix the support $T$ of $f$. In the second case we
allow it to vary. In the following, $C$ will always denote the constant
in Corollary \ref{cs}, and $\mathbb{C}^m$ will always denote the
$m$-dimensional complex Euclidean space equipped with $\ell_2$ norm.
%
\begin{exam}[(Fourier measurements)]\label{csf}
We consider the discrete Fourier transform $\hat f=\Psi f$ where $\Psi$
is a matrix with entries
\[
\Psi_{\omega, t}=\frac1{\sqrt{n}}e^{-i2\pi\omega t/n},\qquad \omega,t\in\{0,\ldots,n-1\}.
\]
We want to reconstruct an $s$-sparse signal $f\in\mathbb{C}^n$ from
linear measurements $\Phi f\in\mathbb{C}^\Omega$, where $\Omega
\subset\{
0,\ldots, n-1\}$ is a uniformly random subset with average cardinality
$k$ and the measurement matrix $\Phi$ is a submatrix of $\Psi$
consisting of random rows with indices in $\Omega$. This is the Fourier
measurement matrix considered in \cite{CRT,CT,RV}. We can\vspace*{1pt} formulate
this random subset precisely using the Bernoulli model. Let\vspace*{1pt} $(\delta
_i)_{i=0}^{n-1}$ be a sequence of independent selectors with $\ez
\delta_i={k}/n$, for $i=0,\ldots, n-1$. Then
\[
\Omega=\{j\dvtx  \delta_j=1\}
\]
and $k=\mathbb{E}|\Omega|$.

Let $y_i$ be the $i$th row of $\Psi$ and $T$ the support of $f$. Write
$y_i^T$ for the restriction of $y_i$ on the coordinate in the set $T$.
For $x,y,z\in\mathbb{C}^n$, we define the tensor $x\otimes y$ as the
rank-one linear operator given by $(x\otimes y)(z)=\langle x,z\rangle
y$. Then
\[
\Phi^*\Phi=\sum_{i\in\Omega}y_i^T
\otimes y_i^T =\sum_{i=0}^{n-1}
\delta_iy_i^T\otimes y_i^T.\vadjust{\goodbreak}
\]
Let $x_j=ny_j^T\otimes y_j^T$. Then
\[
\frac1n\sum_{i=0}^{n-1}x_i =
id_{\mathbb{C}^T}=I_T \quad\mbox{and}\quad \|x_j\|= n
\bigl\|y_j^T\otimes y_j^T\bigr\|=n
\bigl\|y_j^T\bigr\|_2^2\le s.
\]
The next proposition follows easily from Corollary \ref{cs}.
%
\begin{prop}\label{csis}
Assume that the average cardinality of a random set $\Omega$ is
$k=\eps^{-2}s$. Then for $t\eps\le C$,
%
\begin{equation}
\label{csld} \mathbb{P} \Biggl(\Biggl\llVert\frac{n}k\sum
_{i=0}^{n-1}\delta_iy_i^T
\otimes y_i^T-id_{\cz^T}\Biggr\rrVert\ge t\eps
\Biggr)\le se^{-{t^2}/{(2C^2e)}},
\end{equation}
where $\|\cdot\|$ is the operator norm.
\end{prop}
Define
\[
H=id_{\cz^T}-\frac{n}{|\Omega|}\sum_{i=0}^{n-1}
\delta_iy_i^T\otimes y_i^T.
\]
Then $\Phi^*\Phi=\frac{|\Omega|}{n}(I_T-H)$. By the classical Bernstein
inequality, $k/2\le|\Omega|\le3k/2$ with high probability; see
\cite{CT}, Lemma 6.6. Therefore, by choosing $t\eps<1$, we find that the
matrix $I_T-H$ is invertible with high probability. The precise meaning
of ``high probability'' will become clear in a moment. This proposition
is an analog of \cite{CRT}, Theorem 3.1, and \cite{RV}, Theorem 3.3, with
a \textit{single} set $T$. We compare our results with previous results
in the following remark. It is easy to show that $\mathbb{P}(k/2\le
|\Omega|\le3k/2)$ given by Bernstein's inequality dominates
$1-se^{-{t^2}/{(2C^2e)}}$ for the value of $k$ given below. Hence we
only need to consider (\ref{csld}) for the probability of success.
%
\begin{rem}
(i) For a single set $T$ our result is more general than previous
results on the invertibility of $\Phi^*\Phi$ obtained by Candes,
Romberg and Tao in the breakthrough paper \cite{CRT}. In particular, if
we put $t\eps=1/2$ and $\eps^{-2}=8C^2e(M\log n+ \log s)$ for some
$M>0$, then we obtain $k=c_M s \log n$ for some constant $c_M$, and
$I_T-H$ is invertible with probability at least $1- O(n^{-M})$. This
gives \cite{CRT}, Theorem 3.1. Together with \cite{CRT}, Lemma 2.3, or
following verbatim the end of the proof of Theorem 4.2 (\cite{Rau},
Section 7.3), we recover the main results of~\cite{CRT}.

(ii) Allowing arbitrary choices of $k$ and $p$, we recover
\cite{Rau}, Theorem 7.3, and we would like to thank H. Rauhut for bringing
this to our attention. His proof requires considerably more technology.
Both proofs are based on the optimal constant in the noncommutative
Khintchine inequality (used in Rudelson's lemma) which was discovered
independently by the first named author and Pisier; see \cite{Pi2} for
more historic comments. We believe that our proof is more direct.
Moreover, Rauhut established the exact reconstruction results based on
his version of (\ref{csld}) cited above, which shows that an estimate
like (\ref{csld}) is the key to the exact reconstruction problem.
\end{rem}

We now investigate the case with multiple choices of $T$. First, it is
clear that (\ref{csld}) remains valid for polynomially many sets $T$.
In general, we have
%
\begin{equation}
\label{csgen} \mathbb{P} \Biggl(\sup_{|T|\le s}\Biggl\llVert
\frac{n}k\sum_{i=0}^{n-1}
\delta_iy_i^T\otimes y_i^T-id_{\cz^T}
\Biggr\rrVert\ge t\eps\Biggr)\le|S|se^{-
{t^2}/({2C^2e})},
\end{equation}
where $|S|$ denotes the number of set $T$ with $|T|\le s$. Note that
\[
\Delta_s=\inf_{\alpha>0}\sup_{|T|\le s}\biggl\llVert
\alpha\sum_{i\in\Omega
} y_i^T
\otimes y_i^T-id_{\cz^T}\biggr\rrVert.
\]
It follows that
\[
\mathbb{P}(\Delta_s\ge t\eps)\le\mathbb{P} \Biggl(
\sup_{|T|\le
s}\Biggl\llVert\frac{n}k\sum
_{i=0}^{n-1}\delta_iy_i^T
\otimes y_i^T-id_{\cz
^T}\Biggr\rrVert\ge t\eps
\Biggr).
\]

Assume $s\le{n}/2$. Since $|S|\le s{n\choose s}+1\le s(ne/s)^s$, if
%
\begin{equation}
\label{csall} 2\log s+s\log\frac{ne}{s}<\frac{t^2}{2C^2e},
\end{equation}
then with probability at least $1-s^2({ne}/{s})^se^{-{t^2}/{(2C^2e)}}$,
we can recover \textit{all} $s$-sparse signal $f$ from its Fourier
measurements $\Phi f$. From here we are able to obtain different bounds
for $k$ and the corresponding probabilities of success. As an
illustration, we have the following result.
%
\begin{prop}\label{csrip} Assume $s \le{n}/2$.
Let $M>0$ be a precision constant and $n$ be a large integer such that
\[
2\log s+s\log\frac{ne}s<(M+1)s\log\frac{n}s.
\]
Then a random subset $\Omega$ of average cardinality
%
\begin{equation}
\label{cscon} k=8C^2e(M+1)s^2\log\frac{n}s=
c_M s^2\log\frac{n}s
\end{equation}
satisfies RIP with probability at least $1-s^{2}e^s({n}/s)^{-Ms}$.
\end{prop}
\begin{pf}
Put $t\eps=1/2$ in (\ref{csgen}). Since $k=s\eps^{-2}$, we
obtain $t^2=2e(M+1)s\log({n}/s)$. Thanks to the assumption on $n$,
(\ref{csall}) is true. Then
\[
\mathbb{P} \biggl(\Delta_s\ge\frac12 \biggr)\le
s^{2}e^s \biggl(\frac
{n}s \biggr)^{-Ms}.
\]
We have proved the assertion.
\end{pf}
%
\begin{rem}
We can relax the bound for $k$ a little to obtain polynomial
probability of success. Indeed, the same argument as Proposition \ref
{csrip} yields that a random subset $\Omega$ of average cardinality
%
\begin{equation}
\label{cspol} k=8C^2e(M+1)s^2\log{n}= c_M
s^2\log{n}
\end{equation}
satisfies RIP with probability $1-s^{2-s}e^sn^{-Ms}$.
\end{rem}

The good aspect of Proposition \ref{csrip} is that $k$ is linear in
$\log n$. Unfortunately, this is weaker than Rudelson and Vershynin's
results in \cite{RV} $k = O(s\log n \log(s\log n)\log^2 s)$ for fixed
probability $1-\eps$ of success, which was strengthened to
super-polynomially probability of success by Rauhut following their
ideas; see \cite{Rau}. These results are obtained by using deep Banach
spaces techniques. We added our results just for comparison. Of course,
simple applications of Khintchine's inequality are not expected to
replace either majorizing measure techniques or the iterative methods
of \cite{RV} for the uniform estimates required for RIP. It seems known
in the compressed sensing community that the tails bounds alone are not
good enough. To conclude this section, we restate a conjecture on the
best bound of $k$; see \cite{RV} (and \cite{Rau} for further background).
%
\begin{con}
A random subset $\Omega\subset\{0,1,\ldots,n-1\}$ of average cardinality
$k=O(s\log n)$ satisfies RIP with high probability.
\end{con}
\end{exam}

\section*{Acknowledgments}

We would like to thank W. B. Johnson for bringing \cite{PU} to our
attention. After our work was completed, we learned from S. Dirksen
that he also essentially obtained (\ref{rosnc1}) in his Ph.D. thesis
\cite{Di} using a different method in the UIUC analysis seminar on
November 3, 2011. We thank him for helpful conversations. We are also
grateful to the warm response from compressed sensing community.
Especially, we thank K. Lee, H. Rauhut and J. A. Tropp for their
detailed comments on the compressed sensing part of our paper, whose
opinions on credits and earlier results have been incorporated in the
current version.

We thank the anonymous referee for suggestions on improving the
exposition and for bringing Prohorov's inequality to our attention.

Right before this paper is in print, S. Dirksen pointed out that
(\ref{rosnc1}) can also be deduced from \cite{Dir}, Theorem~6.3. Besides, there is new development
on the RIP constant, for which we refer the interested reader to \cite{Rau12}.



\printaddresses

\end{document}